\documentclass{article}

\usepackage{amsmath}
\usepackage{amsthm}
\usepackage{amsxtra}
\usepackage{amsfonts,amssymb}
\newtheorem{Th}{Theorem}
\newtheorem{Prop}[Th]{Proposition}
\newtheorem{Lm}[Th]{Lemma}
\newtheorem{Co}[Th]{Corollary}

\theoremstyle{definition}
\newtheorem{Def}[Th]{Definition}

\newtheorem{Rem}{Remark}

\date{}

\begin{title}
 {\bf ${\rm II}_1$-factor representations of the infinite symmetric inverse semigroup.}
\end{title}
\author{N.I. Nessonov }

\begin{document}
\maketitle

\begin{abstract}
Let  $\mathbb{N}$ be a set of the natural numbers.
 Symmetric inverse semigroup $R_\infty$ is the semigroup of all infinite 0-1 matrices $\left[ g_{ij}\right]_{i,j\in \mathbb{N}}$ with  at most one 1 in each  row and each column such that $g_{ii}=1$ on the complement of a finite set. The binary operation in  $R_\infty$ is the ordinary matrix multiplication.  It is clear that infinite symmetric group $\mathfrak{S}_\infty$ is a subgroup of  $R_\infty$. The map $\star:\left[ g_{ij}\right]\mapsto\left[ g_{ji}\right]$ is an involution on $R_\infty$. We call a function $f$ on  $R_\infty$ positive definite if for all $r_1, r_2, \ldots, r_n\in R_\infty$  the matrix $\left[ f\left( r_ir_j^\star\right)\right]$ is Hermitian and non-negatively definite. A function $f$ said to be  indecomposable if the corresponding $\star$-representation $\pi_f$ is  a factor-representation. A class of the $R_\infty$-central functions (characters) is defined by the  condition  $f(rs)=f(sr)$ for all $r,s\in R_\infty$. In this paper we classify all factor-representations of $R_\infty$ that correspond to the $R_\infty$-central positive definite functions.
\end{abstract}
\section{Introduction}
Let $R_n$ be the set of all $n\times n$ matrices that contain at most one entry of one in each column
and row and zeroes elsewhere. Under matrix multiplication, $R_n$  has the structure of a
semigroup, a set with an associative binary operation and an identity element. The number of ${\rm rank}\,k$ matrices in $R_n$ is ${{\displaystyle{{{n}\choose{k}}}^2}}k!$ and hence $R_n$ has a total of $\sum\limits_{k=0}^n{{\displaystyle{{{n}\choose{k}}}^2}}k!$ elements. Note that the
set of ${\rm rank}\,n$ matrices in the semigroup $R_n$ is isomorphic to $\mathfrak{S}_n$, the symmetric group on $n$
letters.

The semigroup  $R_\infty$ is the inductive limit of the chain $R_n$, $ n =
1, 2, . . .$ , with the natural embeddings: $ R_n\ni r=\left[ r_{ij} \right]\mapsto \hat{r}=\left[ \hat{r}_{ij} \right]\in R_{n+1}$, where $r_{ij}=\hat{r}_{ij}$ for all $i,j\leq n$ and $\hat{r}_{n+1\,n+1}=1$. Respectively, the group $\mathfrak{S}_\infty\subset R_\infty$ is the inductive limit of the chain $\mathfrak{S}_n$, $n=1,2,\ldots$.
 For convenience we will use the matrix representation of the elements of $R_\infty$. Namely, if $r=\left[ r_{ij} \right]\in R_\infty$ then the matrix $\left[ r_{ij} \right]$ contains at most one entry of one in each column
and row and $r_{nn}=1$ for all sufficiently large $n$.
 Denote by $D_\infty\subset R_\infty$ the abelian subsemigroup of the diagonal matrices.
For any subset $\mathbb{A}\subset\mathbb{N}$ denote by $\epsilon_\mathbb{A}$ the matrix $\left[ \epsilon_{ij} \right]\in D_\infty$ such that $\epsilon_{ii} =\left\{
\begin{array}{rl}
0, &\text{ if } i\in \mathbb{A}\\
1, &\text{ if } i\notin \mathbb{A}.
\end{array}\right.$
For example, $\epsilon_{\{2\}}=
\left[\begin{matrix}
1&0&0&\cdots&\\
0&0&0&\cdots\\
0&0&1&\cdots\\
\cdots&\cdots&\cdots&\cdots
\end{matrix}\right]$.

The ordinary transposition of matrices define an involution on $R_\infty:\left[ r_{ij} \right]^{\star}=\left[ r_{ji} \right]$.

Let $\mathcal{B}(\mathcal{H})$ be the algebra of all bounded operators in a Hilbert space $\mathcal{H}$.
By  a $\star$-representation  of $R_\infty$ we mean a homomorphism $\pi$ of $R_\infty$ into the multiplicative semigroup of the algebra $\mathcal{B}(\mathcal{H})$ such that $\pi(r^*)=\left( \pi(r) \right)^*$, where $\left( \pi(r) \right)^*$ is the Hermitian adjoint of operator $\pi(r) $. It follows immediately that $\pi(s)$ is an unitary operator, when $s\in\mathfrak{S}_\infty$, and $\pi(d)$ is self-adjoint projection for $d\in D_\infty$.

 Recall the notion of the quasiequivalent representation.
 \begin{Def}
 Let $\mathcal{N}_1$ and $\mathcal{N}_2$ be the $w^*$-algebras generated by the operators of the representations $\pi_1$ and $\pi_2$, respectively,  of group or  semigroup $G$. $\pi_1 $ and $\pi_2$ are quasiequivalent if there exists isomorphism $\theta:\mathcal{N}_1\to\mathcal{N}_2$ such that $\theta\left(\pi_1(g)\right)=\pi_2(g)$ for all $g\in G$.
 \end{Def}
\begin{Def}\label{support_of_el_semigroup}
Given an element $r=[r_{mn}]\in R_\infty$, let ${\rm supp}\,r$ be the complement of a set $\left\{n\in \mathbb{N}:r_{nn}=1\right\}$.
\end{Def}
 By definition of $R_\infty$, ${\rm supp}\,r$ is a finite set.
Let $c=\left( n_1\;n_2,\;\cdots\;n_k\right)$ be the cycle from $\mathfrak{S}_\infty$. If $\mathbb{A}\subseteq{\rm supp}\,c$, then $q=c\cdot \epsilon_{\mathbb{A}}$ we call as a {\it quasicycle}. Notice, that $\epsilon_{\{k\}}$ is a quasicycle too. Two quasicycles $q_1$ and $q_2$ are called {\it independent}, if $({\rm supp}\,q_1)\cap({\rm supp}\,q_2)=\emptyset$.    Each $r\in R_\infty$ can be decomposed  in the product of the independent  quasicycles\label{quasicycle}:
\begin{eqnarray}\label{decomposition_into_product}
r=q_1\cdot q_2\cdots q_k, \text{ where } {\rm supp}\, q_i\,\cap \,{\rm supp}\, q_j=\emptyset \text{ for all } i\neq j.
\end{eqnarray}
 In general, this decomposition is not unique.

 In this paper we study the $\star$-representations of the $R_\infty$. The main results are the construction of the  list of ${\rm II}_1$-factor representations and the proof of its fullness.

 Finite semigroup $R_n$, its semigroup algebra and the corresponding representations theory was investigated by various authors in \cite{Munn_1, Munn_2, Solomon_1}. The  irreducible representations  of  $R_n$ are indexed by the set of all Young diagrams
with at most $n$ cells. An analog of the Specht modules for finite symmetric  semigroup was built by C. Grood in \cite{Grood}.

 The main motivation of this paper is  due to A. M. Vershik and P. P. Nikitin.  Using the branching rule for the representations of the semigroups $R_n$, they found the full list of the characters on  $R_\infty$ \cite{VN}. We give  very short and simple  proof of the main theorem 2.12 from \cite{VN} in this note. Our elementary approach is based  on the study of the limiting operators, proposed by A. Okounkov \cite{Ok1, Ok2}. We construct the full collection of the new realisations of the corresponding  ${\rm II}_1$-factor-representations in section \ref{Realisation}.
 \subsection{The examples of $\star$-representations of $R_\infty$.} Given $r\in R_\infty$ define in the space $l^2(\mathbb{N})$ the map
 \begin{eqnarray}
 l^2(\mathbb{N})\ni(c_1,c_2,\ldots,c_n,\ldots)\stackrel{\mathfrak{N}(r)}{\mapsto} (c_1,c_2,\ldots,c_n,\ldots)r\in  l^2(\mathbb{N}).
 \end{eqnarray}
It is easy to check that the next statement holds.
\begin{Prop}
  The operators $\mathfrak{N}(r)$ generate the {\it irreducible} $\star$-representation of $R_\infty$.
\end{Prop}
The next important representation  is called  the {\it left regular representation } of $R_\infty$ \cite{APat}. The formula for the action of the corresponding  operators in the space $l^2(R_\infty)$ is given by
\begin{eqnarray}\label{Left_Reg_Repr}
\mathfrak{L}^{reg}(r)\left(\sum\limits_{t\in R_\infty} c_t\cdot t\right)=\sum\limits_{t\in R_\infty:(tt^*)(r^*r)=tt^*} c_t\cdot r\cdot t,\;\;c_t\in\mathbb{C}.
\end{eqnarray}
The operators of the {\it right regular representation } of $R_\infty$ act by
\begin{eqnarray}\label{Right_Reg_Repr}
\mathfrak{R}^{reg}(r)\left(\sum\limits_{t\in R_\infty} c_t\cdot t\right)=\sum\limits_{t\in R_\infty:(t^*t)(r^*r)=t^*t} c_t\cdot t\cdot r^*,\;\;c_t\in\mathbb{C}.
\end{eqnarray}
It is obvious that $\mathfrak{L}^{reg}$ and $\mathfrak{R}^{reg}$ are $\star$-representations of $R_\infty$.

Denote by $l^2_k$ the subspace of $l^2(R_\infty)$ generated by the elements of the view $\sum\limits_{t\in R_\infty:{\rm rank}\,(I-tt^*)=k }c_t\cdot t$. By definition, the subspaces $l^2_k$ are pairwise orthogonal.
It follows from (\ref{Left_Reg_Repr}) and (\ref{Right_Reg_Repr})  that $\mathfrak{L}^{reg}(r)l^2_k\subseteqq l^2_k$ and $\mathfrak{R}^{reg}(r)l^2_k\subseteqq l^2_k$ for all $r\in R_\infty$. Denote by $\mathfrak{L}^{reg}_k$ $\left(\mathfrak{R}^{reg}_k  \right)$ the restriction of $\mathfrak{L}^{reg}$ $\left( \mathfrak{R}^{reg} \right)$ to $l^2_k$.

Since $\mathfrak{R}^{reg}(r_1)\cdot \mathfrak{L}^{reg}(r_2)=\mathfrak{L}^{reg}(r_2)\cdot \mathfrak{R}^{reg}(r_1)$ for all $r_1,r_2\in R_\infty$, the operators $T_k^{(2)}(r_1,r_2)=\mathfrak{L}^{reg}_k(r_1)\cdot\mathfrak{R}^{reg}_k(r_2)$ $\left( r_1,r_2\in R_\infty \right)$ define $\star$-representation of the semigroup $R_\infty\times R_\infty$.
\begin{Prop}
The next properties hold:
\begin{itemize}
  \item {\bf a}) the representation $T_k^{(2)}$ is irreducible for each $k=0,1,\ldots,m,\ldots$;
  \item {\bf b}) the representation $\mathfrak{L}^{reg}_0$ $\left( \mathfrak{R}^{reg}_0\right)$ is ${\rm II}_1$-factor-representation of $R_\infty$; in particular, $\mathfrak{L}^{reg}_0\left( \epsilon_{\{j\}} \right)=0$ $\left( \mathfrak{R}^{reg}_0\left( \epsilon_{\{j\}} \right)=0 \right)$;
  \item {\bf c}) for each $k\geq1$ the representation $\mathfrak{L}^{reg}_k$ $\left( \mathfrak{R}^{reg}_k\right)$ is ${\rm II}_\infty$-factor-representation of $R_\infty$; in particular, if $l>k$ then $\mathfrak{L}^{reg}_k\left( \epsilon_{\{j_1\}}\cdots\epsilon_{\{j_l\}}  \right)=0$ $\left(\mathfrak{R}^{reg}_k\left( \epsilon_{\{j_1\}}\cdots\epsilon_{\{j_l\}}  \right)=0\right)$;
        \item {\bf d})  for each $k\geq 0$ the restriction of $\mathfrak{L}^{reg}_k$ $\left( \mathfrak{R}^{reg}_k\right)$ to the subgroup $\mathfrak{S}_\infty\subset R_\infty$ is ${\rm II}_1$-factor-representation quasiequivalent to the regular one.
\end{itemize}
\end{Prop}
We do not use this proposition below and leave its  proof to the reader.
\subsection{The results}
Let $\pi$ be ${\rm II}_1$-factor $\star$-representation of $R_\infty$ in Hilbert space $\mathcal{H}$. Set $\left(\pi(R_\infty)\right)^\prime=\left\{A\in\mathcal{B}(\mathcal{H}): \pi(r)\cdot A=A\cdot\pi(r) \text{ for all } r\in R_\infty\right\}$, $\left(\pi(R_\infty)\right)^{\prime\prime}=\left(\left(\pi(R_\infty)\right)^\prime\right)^\prime$.
Throughout  this paper we will denote by $\left[v_1,v_2,\ldots, \right]$  the closure of the linear span  of  the vectors $v_1,v_2,\ldots\in \mathcal{H}$. Let ${\rm tr}$ be  the unique faithful normal trace on the factor $\left(\pi(R_\infty)\right)^{\prime\prime}$.
Replacing if needed $\pi$ by the quasi-equivalent  representation, we will  suppose below  that there exists the unit  vector $\xi\in \mathcal{H}$ such that
\begin{eqnarray}
\label{bicyclic_1}  \left[\pi\left(R_\infty \right)\xi \right]=\left[ \left(\pi(R_\infty)\right)^\prime\xi\right]=\mathcal{H};\\
 \label{vector_trace} {\rm tr}(A)=(A\xi,\xi)\text{ for all } A\in \left(\pi(R_\infty)\right)^{\prime\prime}.
\end{eqnarray}
The function $f$ on $R_\infty$, defined by $f(r) =\left(\pi(r)\xi,\xi \right)$, satisfy the next conditions:
\begin{itemize}
  \item ({\bf1}) central, that is, $f(rs)=f(sr)$ for all $s,r\in R_\infty$;
  \item ({\bf2}) positive definite, that is, for all $r_1,r_2, \ldots, r_n\in R_\infty$ the matrix $\left[ f\left(  r_i^*r_j \right) \right]$ is Hermitian and non-negative definite;
  \item ({\bf3}) {\it indecomposable}, that is, it cannot be presented as a sum of two linearly independent functions, satisfying ({\bf1}) and ({\bf2});
      \item ({\bf4}) normalized by $f(e)=1$, where $e$ is unit of $R_\infty$.
\end{itemize}
The functions with  such properties are called the {\it finite characters} of semigroup or group.

From now on, $f$ denotes the character on $R_\infty$.
\begin{Th}[\rm Multiplicativity Theorem]\label{Mult_th} Let $f$ be the {\it indecomposable} character on $R_\infty$, and let $r=q_1\cdot q_2\cdots q_k$ be its decomposition into the product of the independent  quasicycles (see (\ref{decomposition_into_product})). Then $f(r)=\prod\limits_{j=1}^k  f\left(q_j\right)$.
\end{Th}
\begin{proof}
Since $({\rm supp}\,q_k)\cap \left(\bigcup\limits_{j=1}^{k-1}{\rm supp}\,q_j\right)=\emptyset$, there exist the sequence $\left\{s_n \right\}\subset \mathfrak{S}_\infty$ such that
\begin{eqnarray}\label{supp_infinity}
{\rm supp}\,s_nq_ks_n^{-1}\subset(n,\infty ) \text{ and } s_nl=l \text{ for all } l\in\bigcup\limits_{j=1}^{k-1}{\rm supp}\,q_j.
\end{eqnarray}
Let us prove that a sequence $\left\{\pi\left( s_nq_ks_n^{-1}\right)\xi \right\}\subset\mathcal{H}$ converges in the weak  topology to
a vector $f(q_k)\xi$.
Indeed, using Multiplicativity Theorem and (\ref{supp_infinity}), we have $\lim\limits_{n\to\infty}\left(\pi( s_nq_ks_n^{-1})\xi,\eta \right)=\left(\pi\left(q_k \right)\xi,\xi \right)\cdot\left(\xi,\eta \right)$ for any $\eta=\pi(r)\xi$, where $r\in R_\infty$.     Now we conclude from (\ref{bicyclic_1}) that $\lim\limits_{n\to\infty}\left(\pi( s_nq_ks_n^{-1})\xi,\eta \right)=\left(\pi\left(q_k \right)\xi,\xi \right)\cdot\left(\xi,\eta \right)$ for all $\eta\in\mathcal{H}$. Again using (\ref{bicyclic_1}), we obtain that there exists $\lim\limits_{n\to\infty}\pi( s_nq_ks_n^{-1})=f(q_k)\cdot I_\mathcal{H}$ in the weak operator topology. Therefore, $f(r)= f(s_n rs_n^{-1})$ $ \stackrel{(\ref{supp_infinity})}{=} \lim\limits_{n\to\infty}f\left(s_n q_ks_n^{-1}\prod\limits_{j=1}^{k-1} q_j\right)$ $=f(q_k)\cdot f\left(\prod\limits_{j=1}^{k-1} q_j \right)$
   \end{proof}
   Let us consider any cycle  $c=\left(n_1\,n_2\,\ldots\,,n_k \right)$ of the length $k$. If $\mathbb{A}=\left\{n_{j_1},\ldots, n_{j_l} \right\}\subset {\rm supp}\,c$  and $\mathbb{A}\neq\emptyset$ then, using the relations
   \begin{eqnarray*}
  & c=\left(n_1\;n_k \right)\cdots \left(n_1\;n_{j_i} \right)\cdots\left(n_1\; n_2 \right) \text{ and }\\
 &\epsilon_{\{n_1\}}\cdot  \left(n_1\;n_{j_i} \right)\cdot\epsilon_{\{n_1\}}=\epsilon_{\{n_{j_i}\}}\cdot  \left(n_1\;n_{j_i} \right)\cdot\epsilon_{\{n_{j_i}\}} =\epsilon_{\{n_{j_i}\}}\cdot \epsilon_{\{n_1\}},
   \end{eqnarray*}
    we obtain   $f\left(c\cdot\epsilon_\mathbb{A} \right)= f\left( \left(n_1\;n_k \right)\cdots \left(n_1\;n_{j_1} \right)\cdots\left(n_1\; n_2 \right)\cdot\epsilon_{n_{j_1}}\cdot\epsilon_{\mathbb{A}}\right)$ $= f\left(\epsilon_{n_{j_1}}\cdot \left(n_1\;n_k \right)\right.$ $ \left.\cdots \left(n_1\;n_{j_1} \right)\cdots\left(n_1\; n_2 \right)\cdot\epsilon_{n_{j_1}}\cdot\epsilon_{\mathbb{A}}\right)$ $=f\left( \left(n_1\;n_k \right) \cdots\epsilon_{n_{j_1}}\cdot \left(n_1\;n_{j_1} \right)\cdot\epsilon_{n_{j_1}}\cdots\left(n_1\; n_2 \right)\cdot\epsilon_{\mathbb{A}}\right)$ $=f\left( \left(n_1\;n_k \right) \cdots\left(n_1\;n_{j_1+1} \right)\cdot\epsilon_{n_1}\cdot\epsilon_{n_{j_1}}\cdot\left(n_1\;n_{j_1-1} \right)\cdots\left(n_1\; n_2 \right)\cdot\epsilon_{\mathbb{A}}\right)$ $= f\left(\epsilon_{n_{j_1+1}}\cdot \left(n_1\;n_k \right) \right.$ $\left. \cdots\left(n_1\;n_{j_1+1} \right)\cdot\left(n_1\;n_{j_1-1} \right)\cdots\left(n_1\; n_2 \right)\cdot\epsilon_{\mathbb{A}}\right)$ $=f\left( \widetilde{c}\cdot\epsilon_{n_{j_1+1}}\cdot\epsilon_{\mathbb{A}}\right)$, where $\widetilde{c}=\left(n_1\;\cdots\; n_{j_1-1} \;n_{j_1+1}\;\cdots\;n_k \right)$. Applying these equalities  $k$ times, we obtain $f\left(c\cdot\epsilon_\mathbb{A} \right)=f\left(\epsilon_{{\rm supp}\,c} \right)$.

   Therefore, the next corollary is the supplement to theorem  \ref{Mult_th} and does not need of the proof already.
\begin{Co}\label{Corollary}
There exists $\rho\in[0,1]$ such that
   for any quasicycle $q=c\cdot\epsilon_\mathbb{A}$, where $\mathbb{A}\neq\emptyset$, we have $f\left( q\right)=\rho^{\#({{\rm supp}\,c})}$.
\end{Co}

Next statement follows from theorem \ref{Mult_th} and \cite{Thoma}.
\begin{Prop}\label{restriction_to_symmetric}
The restriction of $f$ to the symmetric subgroup $\mathfrak{S}_\infty\subset R_\infty$ is a character of $\mathfrak{S}_\infty$. Denote by $\alpha=\left\{\alpha_1\geq\alpha_2\geq\ldots>0 \right\}$ and $\beta=\left\{\beta_1\geq\alpha_2\geq\ldots>0\right\}$  the corresponding {\it Thoma parameters}.
Let $s$ be the permutation from $\mathfrak{S}_\infty$ and let $s=c_1\cdot c_2 \cdots c_k$ be its cycle decomposition, where the length  of the cycle $c_j$ equal $l(c_j)>1$. Then
\begin{eqnarray*}
f(c_j)=\sum\limits_i \alpha_i^{l(c_j)}+(-1)^{l(c_j)}\sum\limits_i\beta_i^{l(c_j)}\text{ \rm and } f(s)=\prod\limits_{j=1}^k f\left(c_j \right).
\end{eqnarray*}
Further, we denote this restriction by $\chi_{\alpha\beta}$.
\end{Prop}

 The main result of this paper is the following theorem
 \begin{Th}[Main Theorem]\label{main_th}
 Under the notations of Corollary \ref{Corollary} and Proposition \ref{restriction_to_symmetric}, we have $\rho\in \alpha\cup 0$. If $\rho=\lambda$, $c= (1\,2\,\ldots, \,k\}$ and  $\mathbb{A}\subset{\rm supp}\,c= \{1,2\ldots, k\}$ then $f\left(c\epsilon_{\mathbb{A}}\right)=\left\{
 \begin{array}{rl}
\lambda^k,&\text{ if } \mathbb{A}\neq \emptyset\\
 \chi_{\alpha\beta}(c),&\text{ if } \mathbb{A}= \emptyset.
 \end{array}\right.$
\end{Th}
\section{The proof of the main theorem}
Let us recall first  the important statement from \cite{Ok1, Ok2} that we will use below. 
\begin{Lm}\label{OkUnkov_operator}
 For any $k$ the sequence $\left\{\pi\left(\left(k\;n \right) \right) \right\}_{n=1}^\infty$ converges in the weak operator topology to self-adjoint operator $\mathcal{O}_k\in\pi\left(\mathfrak{S}_\infty \right)^{\prime\prime}\subset\pi\left(R_\infty \right)^{\prime\prime}$.
 \end{Lm}
 \begin{proof}
 To prove, it is suffices to notice that $\left(\pi\left(\left(k\;n+1 \right) \right)\cdot\pi(r_1)\xi,\pi(r_2)\xi \right)$ $=\left(\pi\left(\left(k\;N \right) \right)\cdot\pi(r_1)\xi,\pi(r_2)\xi \right)$ for all $r_1, r_2\in R_\infty$,  $N>n$ and apply (\ref{bicyclic_1}).
 \end{proof}
 \begin{Lm}
 Let $S(\mathcal{O}_k)$ be the spectrum of operator $\mathcal{O}_k$ and let $\mu$ denotes the spectral measure of operator $\mathcal{O}_k$, corresponding to vector $\xi$. Then the following hold:
 \begin{itemize}
   \item {\bf 1}) the measure $\mu$  is discrete  and its atoms  can only  accumulate to zero;
   \item  {\bf 2}) if $\;\mathcal{O}_k=\sum\limits_{\lambda\in S\left(\mathcal{O}_k \right)} \lambda E_k(\lambda)$ is the spectral decomposition of  $\mathcal{O}_k$ then $\left(E_k(\lambda)\xi,\xi \right)=m(\lambda)\cdot|\lambda|$, where $m(\lambda)\in\mathbb{N}\cup0$;
   \item   {\bf 3}) if $\lambda\in S(\mathcal{O}_k)$ is positive (negative) then there exists some Thoma parameter (see Proposition \ref{restriction_to_symmetric})  such that $\lambda$ $=\alpha_k$ $\left(\lambda=-\beta_k \right)$ and $m(\lambda)=\#\left\{k:\alpha_k=\lambda \right\}\;\;\left(\#\left\{k:-\beta_k=\lambda \right\} \right)$.
 \end{itemize}
 \end{Lm}
 \begin{Lm}\label{commutativity-lemma}
 The operators $\mathcal{O}_j$ and $\pi(\epsilon_{\{k\}})$ mutually commute.
 \end{Lm}
 \begin{proof}
In the case $k\neq j$ lemma is obvious. By (\ref{bicyclic_1}), it is sufficient to show that
\begin{eqnarray}\label{commutativity}
\left(\pi(\epsilon_{\{k\}})\cdot\mathcal{O}_k\xi,\pi(r)\xi \right)=\left(\mathcal{O}_k\cdot\pi(\epsilon_{\{k\}})\xi,\pi(r)\xi \right)\;\text{ for all  } \; r\in R_\infty.
\end{eqnarray}
Fix the naturale number $N(r)$ such that $r\in R_{N(r)}$. For any sufficiently large number $N$ the next chain of the equalities holds:
\begin{eqnarray*}
&\left(\pi(\epsilon_{\{k\}})\cdot\mathcal{O}_k\xi,\pi(r)\xi \right)\stackrel{\text{Lemma }\ref{OkUnkov_operator}}{=}\lim\limits_{n\to\infty}\left(\pi(\epsilon_{\{k\}})\cdot\pi\left(\left(k\;n\right) \right)\xi,\pi(r)\xi \right)\\
&\stackrel{N>{\rm max}\{k,N(r)\}}{=}\left(\pi(\epsilon_{\{k\}})\cdot\pi\left(\left(k\;N\right) \right)\xi,\pi(r)\xi \right)=\left(\pi(\epsilon_{\{k\}})\cdot\pi\left(\left(k\;N\right) \right)\cdot(\pi(r))^*\xi,\xi \right)\\
&\left(\pi\left(\left(k\;N\right) \right)\cdot\pi(\epsilon_{\{N\}})\cdot(\pi(r))^*\xi,\xi \right)\stackrel{N>N(r)}{=}\left(\pi\left(\left(k\;N\right) \right)\cdot (\pi(r))^*\cdot\pi(\epsilon_{\{N\}})\xi,\xi \right)\\
&=\left(\pi(\epsilon_{\{N\}})\cdot\pi\left(\left(k\;N\right) \right)\cdot (\pi(r))^*\xi,\xi \right)=\left(\pi(\epsilon_{\{N\}})\cdot\pi\left(\left(k\;N\right) \right) \xi,\pi(r)\xi \right)\\
&=\left(\pi\left(\left(k\;N\right) \right)\cdot\pi(\epsilon_{\{k\}}) \xi,\pi(r)\xi \right)=\left(\mathcal{O}_k\cdot\pi(\epsilon_{\{k\}})\xi,\pi(r)\xi \right).
\end{eqnarray*}
The equality (\ref{commutativity}) is proved.
\end{proof}
\begin{Lm}
Let $\mathfrak{A}_j$  be $w^*$-algebra, generated by the operators $\mathcal{O}_j$ and $\pi(\epsilon_{\{j\}})$. Then the following hold:
\begin{itemize}
  \item {\rm i})  $\pi(\epsilon_{\{j\}})$ is a minimal projection in $\mathfrak{A}_j$;
  \item {\rm ii}) if  $\lambda\leq 0$ then $E_j(\lambda)\cdot\pi(\epsilon_{\{j\}})=0$.
\end{itemize}
\end{Lm}
\begin{proof}
To prove the property {\rm i}), we notice that
\begin{eqnarray}\label{relations_generators}
\pi(\epsilon_{\{j\}})\cdot\pi\left(\left(j\;n \right) \right)\cdot\pi(\epsilon_{\{j\}})=\pi(\epsilon_{\{j\}})\cdot\pi(\epsilon_{\{n\}}).
\end{eqnarray}
Since $\pi$ is ${\rm II}_1$-factor representation, the limit of the sequence  $\left\{ \pi(\epsilon_{\{n\}})\right\}$ exists in the weak operator topology. Namely,
\begin{eqnarray*}
w^*{\text{-}}\lim\limits_{n\to\infty}\pi(\epsilon_{\{n\}})=\kappa\cdot I, \text{ where } \kappa=\left(\pi(\epsilon_{\{1\}})\xi,\xi \right).
\end{eqnarray*}
Hence, applying (\ref{relations_generators}),  lemma \ref{OkUnkov_operator}, lemma \ref{commutativity-lemma} and passing to the limit $n\to\infty$, we obtain
\begin{eqnarray}\label{need_for_zero}
\mathcal{O}_j\cdot\pi(\epsilon_{\{j\}})=\pi(\epsilon_{\{j\}})\cdot\mathcal{O}_j\cdot\pi(\epsilon_{\{j\}})=\kappa\cdot\pi(\epsilon_{\{j\}}).
\end{eqnarray}
Therefore, $\mathcal{O}_j^m\cdot\pi(\epsilon_{\{j\}})=\kappa^m\cdot \pi(\epsilon_{\{j\}})$ for all  $m\in\mathbb{N}$. Property {\rm i}) is proved.

We now come to the proof of  {\rm ii}).

By property {\rm i}) and lemma \ref{commutativity-lemma}, in the case, when $\pi(\epsilon_{\{j\}})\neq 0$, there exists only one spectral projection, for example $E_j(\hat{\lambda})$, such that
\begin{eqnarray*}
E_j(\hat{\lambda})\cdot\pi(\epsilon_{\{j\}})= \pi(\epsilon_{\{j\}}) \text{ and  } E(\lambda^\prime)\cdot\pi(\epsilon_{\{j\}})=0\; \text{ for all } \lambda^\prime\neq\hat{\lambda}.
\end{eqnarray*}
Hence, under the condition $\hat{\lambda}\neq0$, we obtain
\begin{eqnarray*}
&\hat{\lambda}\left(\pi(\epsilon_{\{j\}})\xi,\xi  \right)=\left(\mathcal{O}_j\cdot\pi(\epsilon_{\{j\}})\xi,\xi \right)
\stackrel{\text{Lemma }\ref{OkUnkov_operator}}{=}\lim\limits_{n\to\infty}\left(\pi((j\; n))\cdot \pi(\epsilon_{\{j\}})\xi,\xi\right)\\
&\stackrel{N\neq j}{=}\left(\pi((j\; N))\cdot \pi(\epsilon_{\{j\}})\xi,\xi \right)=\left(\pi(\epsilon_{\{j\}})\cdot\pi((j\; N))\cdot \pi(\epsilon_{\{j\}})\xi,\xi \right)\\
&=\left(\pi(\epsilon_{\{j\}})\cdot \pi(\epsilon_{\{N\}})\xi,\xi \right)\stackrel{\text{Theorem \ref{Mult_th}}}{=}\left(\pi(\epsilon_{\{j\}})\xi,\xi \right)\left( \pi(\epsilon_{\{N\}})\xi,\xi \right)=\left(\pi(\epsilon_{\{j\}})\xi,\xi \right)^2.
\end{eqnarray*}
Since $\pi(\epsilon_{\{j\}})\neq 0$, then, by (\ref{bicyclic_1}), $\left(\pi(\epsilon_{\{j\}})\xi,\xi \right)\neq0$. Therefore,
\begin{eqnarray}\label{nonzero}
\hat{\lambda}=\left(\pi(\epsilon_{\{j\}})\xi,\xi \right)>0.
\end{eqnarray}
If $\hat{\lambda}=0$, i.e. $\pi(\epsilon_{\{j\}})\leq E_k(0)$, then, using (\ref{need_for_zero}), we obtain $\left(\pi(\epsilon_{\{j\}})\xi,\xi \right)=0$.
Thus, by (\ref{bicyclic_1}), $\pi(\epsilon_{\{j\}})=0$.
\end{proof}
The next statement follows from preceding lemma and lemma \ref{commutativity-lemma}.
\begin{Co}\label{co_main}
  If $f\left(\epsilon_{\{k\}} \right)\neq0$ then there exists positive $\lambda\in S\left(\mathcal{O}_k \right)$ such that $\pi(\epsilon_{\{k\}})\leq E_k(\lambda)$ and $f\left(\epsilon_{\{k\}}\right)=\lambda$.
\end{Co}
\subsection{The proof of Theorem \ref{main_th}}\label{proof_main_th}
By theorem \ref{Mult_th} and proposition \ref{restriction_to_symmetric}, it is sufficient to find the value of the character $f$ on quasicycle $q=c\cdot\epsilon_\mathbb{A}$, where $c=(1\;2\;\ldots\;k)$, $\mathbb{A}\subset{\rm supp}\,c= \{1,2\ldots, k\}$ and $\mathbb{A}\neq\emptyset$. Without loss of generality we can assume that $\pi(\epsilon_{\{k\}})\neq0$.

Define the map $\vartheta_{\mathbb{A}}: \left\{1,2,\ldots,k \right\}\mapsto \left\{e,\epsilon_{\{1\}} \right\}$, where $e$ is unit of  $R_\infty$, by
\begin{eqnarray}
\vartheta_\mathbb{A}(j)=\left\{
 \begin{array}{rl}
 \epsilon_{\{1\}},&\text{ if } j\in\mathbb{A}\\
 e,&\text{ if }  j\notin\mathbb{A}.
 \end{array}\right.
\end{eqnarray}
Since $c=(1\;k)\cdot(1\;k-1)\cdots (1\;2)$, we have
\begin{eqnarray*}
c\cdot\epsilon_{\mathbb{A}}=\vartheta_{\mathbb{A}}(k)\cdot(1\;k)\cdot\vartheta_{\mathbb{A}}(k-1)\cdot(1\;k-1)\cdots\vartheta_{\mathbb{A}}(2)\cdot(1\;2)\cdot\vartheta_{\mathbb{A}}(1).
\end{eqnarray*}
Therefore, for any collection $\left\{n_2< n_3<\ldots< n_k \right\}\subset\mathbb{N}$, where $n_2>k$,  we obtain
\begin{eqnarray*}
f\left(c\cdot\epsilon_{\mathbb{A}} \right)=\left(\pi\left(\vartheta_{\mathbb{A}}(k)\cdot(1\;k)\cdot\vartheta_{\mathbb{A}}(k-1)\cdot(1\;k-1)\cdots\vartheta_{\mathbb{A}}(2)\cdot(1\;2) \cdot\vartheta_{\mathbb{A}}(1)\right)\xi,\xi\right)\\
=\left(\pi\left(\vartheta_{\mathbb{A}}(k)\cdot(1\;n_k)\cdot\vartheta_{\mathbb{A}}(k-1)\cdot(1\;n_{k-1})\cdots\vartheta_{\mathbb{A}}(2)\cdot(1\;n_2)\cdot\vartheta_{\mathbb{A}}(1) \right)\xi,\xi\right),
\end{eqnarray*}
Passing in series to the limits $n_k\to\infty, n_{k-1}\to\infty,\ldots,n_2\to\infty$, we come to the relation
\begin{eqnarray*}
\begin{split}
f\left(c\cdot\epsilon_{\mathbb{A}} \right)\stackrel{\text{Lemma \ref{OkUnkov_operator}}}{=}\left(\pi\left(\vartheta_{\mathbb{A}}(k) \right)\cdot\mathcal{O}_1\cdot \pi\left(\vartheta_{\mathbb{A}}(k-1) \right)\cdot\mathcal{O}_1\cdots \pi\left(\vartheta_{\mathbb{A}}(2) \right)\cdot\mathcal{O}_1\cdot\pi\left(\vartheta_{\mathbb{A}}(1) \right)\xi,\xi\right)\\
\stackrel{\text{Lemma \ref{commutativity-lemma} }}=\left\{
 \begin{array}{rl}
\left(\pi\left(\epsilon_{\{1\}} \right)\mathcal{O}_1^{k-1}\xi,\xi \right),&\text{ if } \mathbb{A}\neq \emptyset\\
 \left(\mathcal{O}_1^{k-1}\xi,\xi \right),&\text{ if } \mathbb{A}= \emptyset.
 \end{array}\right.
 \stackrel{{\text{Corollary \ref{co_main}}}}{=}\left\{
 \begin{array}{rl}
\lambda^k,&\text{ if } \mathbb{A}\neq \emptyset\\
 \chi_{\alpha\beta}(c),&\text{ if } \mathbb{A}= \emptyset.
 \end{array}\right.
\end{split}
\end{eqnarray*}
Theorem \ref{main_th} is proved.
\section{The realisations of ${\rm II}_1$-factor-representations}\label{Realisation}
Our aim in this section is the construction of  ${\rm II}_1$factor-representations for $R_\infty$.
\subsection{The parameters of the ${\rm II}_1$-factor-representations.}\label{parameters_of_repr}
Let $B(\mathbb{\mathbf{H}})$ denote
the set of all (bounded linear) operators acting on the complex Hilbert space $\mathbf{H}$.
Let ${\rm Tr}$ be the ordinary\footnote{ ${\rm Tr}(\mathfrak{p})$=1 for any  nonzero minimal projection  $\mathfrak{p}\in B(\mathbb{\mathbf{H}})$. } trace on $B(\mathbb{\mathbf{H}})$. Fix the self-adjoint operator $A\in B(\mathbf{H})$ and the {\it minimal} orthogonal projection $\mathbf{q}\in B(\mathbf{H})$ such that $A\cdot\mathbf{q}=\mathbf{q}\cdot A$. Let ${\rm Ker}\, A=\left\{ u\in B(\mathbf{H}):Au=0\right\}$, and let  $({\rm
      Ker} A)^\perp=\mathbf{H}\ominus {\rm Ker}\, A $. Denote by $E(\Delta)$ the spectral projection of operator $A$, corresponding to $\Delta\subset \mathbb{R}$.  Suppose that occur the  next conditions:
\begin{itemize}
  \item {\rm(a)} ${\rm Tr} (|A|)\leq1$,
   \item {\rm(b)} if $\mathbf{q}\neq0$, then   ${\rm Tr}(\mathbf{q})=1$;
  \item {\rm(c)} if ${\rm Tr} (|A|)=1$ then ${\rm Ker}\, A=0$;
    \item {\rm(d)} if ${\rm Tr} (|A|)<1$ then
      $\dim ({\rm Ker}\, A)=\infty$;
       \item {\rm(e)} $\mathbf{q}\cdot E([-1,0])=0$;
\end{itemize}
\subsection{Hilbert space $\mathcal{H}_A^\mathbf{q}$.}\label{hp}
Let $\mathbb{S}=\left\{ 1,2,\ldots, {\rm dim}\,\mathbf{H} \right\}$.
Fix the matrix unit $\left\{ e_{kl}\right\}_{k,l\in\mathbb{S}}\subset B(\mathbf{H})$. Suppose for the convenience that
\begin{eqnarray*}
Ae_{ll}=e_{ll}A \text{ for all } l.
\end{eqnarray*}
Let $\mathbb{S}_{reg}=\left\{ n_1,n_2,\ldots\right\}=\left\{ l:e_{ll}\mathbf{H}\subset {\rm Ker}\, A\right\}$ \label{s_regular}(see ({\rm d})), where $n_k<n_{k+1}$.
 Define
a state $\psi_k$ on $B\left(\mathbf{H} \right)$ as
follows
\begin{eqnarray}\label{psik}
\psi_k\left(b \right)={\rm Tr}\left(b|A| \right)+\left(1-{\rm
Tr}\left(|A| \right)\right) {\rm Tr}\left(b e_{n_kn_k}\right),\;\;b\in B\left(\mathbf{H} \right).
\end{eqnarray}
Let $ _1\psi_k$ denote the product-state on $B\left(\mathbf{H}\right)^{\otimes k}$:
\begin{eqnarray}\label{product_psik}
 _1\psi_k\left(b_1\otimes b_2\otimes \ldots\otimes b_k
 \right)=\prod\limits_{j=1}^k\psi_j\left(b_j \right).
\end{eqnarray}
Now define inner product on $B \left( \mathbf{H}\right)^{\otimes k}$ by
\begin{eqnarray}\label{inner_product_psik}
\left( v,u\right)_k=\,_1\psi_k\left( u^*v \right).
\end{eqnarray}
Let $\mathcal{H}_k$ denote the Hilbert space obtained by completing
$B \left( \mathbf{H}\right)^{\otimes k}$  in above inner product
norm. Now we consider the natural isometrical embedding
\begin{eqnarray*}
v\ni\mathcal{H}_k\mapsto v\otimes {\rm I}\in\mathcal{H}_{k+1},
\end{eqnarray*}
and define Hilbert space $\mathcal{H}^\mathbf{q}_A$ as completing
$\bigcup\limits_{k=1}^\infty \mathcal{H}_k$.
\subsection{The action of $R_\infty$ on $\mathcal{H}^\mathbf{q}_A$. }\label{action}
 First, using the
embedding\\ $a\ni B\left(\mathbf{H} \right)^{\otimes
k}\mapsto a\otimes{\rm I}\in B\left(\mathbf{H}
\right)^{\otimes (k+1)}$, we identify $B\left(\mathbf{H}
\right)^{\otimes k}$ with subalgebra $B\left(\mathbf{H}
\right)^{\otimes
k}\otimes\mathbb{C}\subset B\left(\mathbf{H}
\right)^{\otimes (k+1)}$. Therefore, algebra
$B\left(\mathbf{H}\right)^{\otimes\infty}=\bigcup\limits_{n=1}^\infty
B\left(\mathbf{H} \right)^{\otimes n}$ is well defined.

Now we construct the explicit embedding of $\mathfrak{S}_\infty$ into the unitary subgroup of  $B\left(\mathbf{H}\right)^{\otimes\infty}$. For $a\in B\left(\mathbf{H}\right)$ put $a^{(k)}={\rm I}\otimes\cdots\otimes{\rm I}\otimes\underbrace{a}_{k}\otimes{\rm I}\otimes{\rm I}\cdots$. Let $E_{-}=E([-1,0[)$ and let
\begin{eqnarray*}
U_{E_{-}}^{(k,\,k+1)}=({\rm I}-E_{-})^{(k)}({\rm I}-E_{-})^{(k+1)}+E_{-}^{(k)} ({\rm
I}-E_{-})^{(k+1)}\\
+ ({\rm I}-E_{-})^{(k)}E_{-}^{(k+1)}-E_{-}^{(k)}E_{-}^{(k+1)}.
\end{eqnarray*}
Define the unitary operator $T\left((k\;k+1)\right)\in B\left(\mathbf{H}
\right)^{\otimes\infty}$ as follows
\begin{eqnarray}\label{embeding_T}
T\left((k\;k+1)\right)=U_{E_{-}}^{(k,\, k+1)}\sum\limits_{ij}e_{ij}^{(k)}e_{ji}^{(k+1)}.
\end{eqnarray}
Put $T(\epsilon _{\{1\}})=\mathbf{q}^{(1)}$.
Using the relation $(k\;k+m)=(k+m-1\;k+m)$ $\cdots(k $ $+1\;k+2)$ $\cdot(k\;k+1)\cdot(k+1\;k+2)\cdots(k+m-1\;k+m)$,
 we can to prove that
\begin{eqnarray}\label{any_transposition}
\begin{split}
T\left((k\;k+m)\right)=U_{E_{-}}^{(k,\, k+m)}\sum\limits_{ij}e_{ij}^{(k)}e_{ji}^{(k+m)}, \text{ where }\\
U_{E_{-}}^{(k,\, k+m)}=({\rm I}-E_{-})^{(k)}({\rm I}-E_{-})^{(k+m)}-E_{-}^{(k)}E_{-}^{(k+m)}\\
+\left(({\rm I}-E_{-})^{(k)}E_{-}^{(k+m)}+E_{-}^{(k)}({\rm I}-E_{-})^{(k+m)}\right)\prod\limits_{j=k+1}^{k+m-1}\left({\rm I}-2E_- \right)^{(j)}.
\end{split}
\end{eqnarray}
A easy verification of the standard relations between $\left\{ T((k\;k+1))\right\}_{k\in\mathbb{N}}$ and
$\mathbf{q}^{(1)}$ shows that
$T$ extends by multiplicativity to the $\star$-homomorphism   of $R_\infty$ to $B\left(\mathbf{H}
\right)^{\otimes\infty}$.

Left multiplication in $B\left(\mathbf{H}\right)^{\otimes\infty}$ defines  $\star$-representation $\mathfrak{L}_A$ of   $B\left(\mathbf{H}\right)^{\otimes\infty}$ by bounded operators on  $\mathcal{H}^\mathbf{q}_A$. Put $\Pi_A(r)=\mathfrak{L}_A\left( T(r)\right)$, $r\in R_\infty$. Denote by $\pi_A^{(0)}$ the restriction of $\Pi_A$ to $\left[ \Pi_A\left( R_\infty\right)\xi _0\right]$, where $\xi _0$ is the vector from $\mathcal{H}^\mathbf{q}_A$ corresponding to the unit element of $B\left(\mathbf{H}\right)^{\otimes\infty}$.
\begin{Rem}
If $T\left( \epsilon_{\{1\}} \right)=0$ and $T(s)$ $(s\in\mathfrak{S}_\infty)$ is defined by (\ref{embeding_T}), then $\left\{  \pi_A^{(0)}\left( R_\infty \right) \right\}^{\prime\prime}=\left\{  \pi_A^{(0)}\left( \mathfrak{S}_\infty \right) \right\}^{\prime\prime}$ and the corresponding representation $\pi_A^{(0)}$ is ${\rm II}_1$-factor-representation of $R_\infty$.

\end{Rem}

\begin{Rem}
If $A=\mathbf{p}$, where $\mathbf{p}$ is one-dimensional projection, then $\left(\Pi_\mathbf{p}(s)\xi _0,\xi _0 \right)_{\mathcal{H}^\mathbf{q}_A}=1$ for all  $s\in\mathfrak{S}_\infty$.  Therefore, we obtain two  corresponding representations:
\begin{itemize}
  \item {\bf 1}) $\pi_\mathbf{p}^{(0)}(s)=I$ for all $s\in\mathfrak{S}_\infty$, $\pi_\mathbf{p}^{(0)}(\epsilon_{\{k\}})=I$;
  \item {\bf 2}) $\pi_\mathbf{p}^{(0)}(s)=I$ for all $s\in\mathfrak{S}_\infty$, $\pi_\mathbf{p}^{(0)}(\epsilon_{\{k\}})=0$.
\end{itemize}
If $A=-\mathbf{p}$, then we have the unique representation: $\pi_{-\mathbf{p}}^{(0)}(s)=({\rm sign}\,s)\cdot I$ for all $s\in\mathfrak{S}_\infty$ and $\pi^{(0)}_{-\mathbf{p}}(\epsilon_{\{k\}})=0$.
\end{Rem}
\subsection{The character formula}
Let $f(r)=\left(\pi_A^{(0)}(r)\xi_0,\xi_0 \right)$, $r\in R_\infty$. Here we will find a formula for $f$. Next statement follows from (\ref{any_transposition}), by the ordinary calculation.
\begin{Lm}\label{Realisations_okounkov}
Let $E_-\neq I$. Then  $\lim\limits_{n\to\infty}\pi_A^{(0)}((k\;n))=\mathfrak{L}_A\left( A^{(k)} \right)$ in the weak operator topology.
\end{Lm}
It follows from definition of $\pi_A^{(0)}$ that $f$ satisfies to theorem \ref{Mult_th}. Therefore, it is sufficient to find the value $f$ on quasicycle $q=c\epsilon_\mathbb{A}$, where $c=(1\;2\;\ldots\;k)$ and $\mathbb{A}\subset{\rm supp}\,c$.

As in the proof of theorem \ref{main_th}(section \ref{proof_main_th},  lemma \ref{Realisations_okounkov} gives
\begin{eqnarray*}
f\left(c\cdot\epsilon_{\mathbb{A}}\right) =
 \left\{\begin{array}{rl}
{\rm Tr}\left(|A|\cdot\mathbf{q}\cdot A^{k-1} \right),&\text{ if } \mathbb{A}\neq \emptyset\\
{\rm Tr}\left( |A|\cdot A^{k-1} \right) ,&\text{ if } \mathbb{A}= \emptyset.
 \end{array}\right.
\end{eqnarray*}
It follows from p. \ref{parameters_of_repr} that there exists a spectral projection $E(\lambda)$  of operator $A$, where $\lambda>0$, with the property $\mathbf{q}\cdot E(\lambda)=\mathbf{q}$. Hence, using proposition \ref{restriction_to_symmetric}, we obtain
\begin{eqnarray*}
f\left(c\cdot\epsilon_{\mathbb{A}}\right) =
 \left\{\begin{array}{rl}
\lambda^k,&\text{ if } \mathbb{A}\neq \emptyset\\
\chi_{\alpha\beta}(c) ,&\text{ if } \mathbb{A}= \emptyset.
 \end{array}\right.
\end{eqnarray*}
{}
B. Verkin ILTPE of NASU - B.Verkin Institute for Low Temperature Physics and Engineering
of the National Academy of Sciences of Ukraine

 n.nessonov@gmail.com
\end{document}